\documentclass{article}
 \usepackage[usenames]{color}
\usepackage{amsthm}
\usepackage{amssymb}
\usepackage{latexsym}
\usepackage{euscript}
\usepackage{dsfont}
\usepackage[all]{xy}
\usepackage{hyperref}
\usepackage{graphicx}

\theoremstyle{plain}
\newtheorem{theorem}{Theorem}[section] 
\newtheorem{lemma}[theorem]{Lemma}     
\newtheorem{corollary}[theorem]{Corollary}
\newtheorem{proposition}[theorem]{Proposition}

\theoremstyle{definition}
\newtheorem{assumption}{Assumption} 
\newtheorem{definition}[theorem]{Definition}

\theoremstyle{remark}
\newtheorem{remark}[theorem]{Remark}

\newcommand{\defn}[1]{\emph{#1}}

\newcommand{\ie}{i.e.\ }
\newcommand{\eg}{e.g.\ }

\newcommand{\N}{\mathbb{N}}
\newcommand{\Z}{\mathbb{Z}}

\newcommand{\R}{\mathbb{R}}
\newcommand{\C}{\mathbb{C}}
\newcommand{\U}{\mathbb{H}}
\renewcommand{\P}{\mathbb{P}}

\newcommand{\cat}[1]{\mathcal{#1}}
\newcommand{\perv}[3]{\cat{P}^{#1}_{#3}{[#2]}}

\newcommand{\constr}[1]{\cat{D}^\textrm{c}(#1)}
\newcommand{\constran}[1]{\cat{D}^{\textrm{$\R$-an-c}}(#1)}
\newcommand{\derfuk}[1]{\cat{DAF}(#1)}

\newcommand{\id}{id}

\newcommand{\mor}[2]{{\mathrm{Hom}}(#1,#2)}
\newcommand{\ext}[3]{{\mathrm{Ext}^{#1}}(#2,#3)}
\newcommand{\Ext}[4]{{\mathrm{Ext}_{#1}^{#2}}(#3,#4)}

\newcommand{\stab}[1]{\mathrm{Stab}(#1)}
\newcommand{\stabo}[1]{\mathrm{Stab}_0(#1)}

\title{Stability conditions, torsion theories and tilting} 

\author{Jonathan Woolf}

\begin{document}
\maketitle

\begin{abstract}
The space of stability conditions on a triangulated category is naturally partitioned into subsets $U(\cat{A})$ of stability conditions with a given heart $\cat{A}$. If $\cat{A}$ has finite length and $n$ simple objects then $U(\cat{A})$ has a simple geometry, depending only on $n$. Furthermore, Bridgeland has shown that if $\cat{B}$ is obtained from $\cat{A}$ by a simple tilt, \ie by tilting at a torsion theory generated by one simple object, then the intersection of the closures of $U(\cat{A})$ and $U(\cat{B})$ has codimension one. 

Suppose that $\cat{A}$, and any heart obtained from it by a finite sequence of (left or right) tilts at simple objects, has 
finite length and 
finitely many indecomposable objects. Then we show that the closures of $U(\cat{A})$ and $U(\cat{B})$ intersect if and only if $\cat{A}$ and $\cat{B}$ are related by a tilt, and that the dimension of the intersection can be determined from the torsion theory. In this situation the union of subsets $U(\cat{B})$, where $\cat{B}$ is obtained from $\cat{A}$ by a finite sequence of simple tilts, forms a component of the space of stability conditions. We illustrate this by computing (a component of) the space of stability conditions on the constructible derived category of the complex projective line stratified by a point and its complement.  
\end{abstract}


\section{Introduction}

Bridgeland introduced the notion of a stability condition on a triangulated category $\cat{C}$ in \cite{MR2373143} and proved that the set of locally-finite stability conditions forms a complex manifold $\stab{\cat{C}}$. This space provides a geometric picture of many aspects of the category and carries a natural action of its automorphisms. Each stability condition determines an abelian heart in the triangulated category. If we denote  the subset of locally-finite stability conditions with heart $\cat{A}$ by  $U(\cat{A})$ then the $U(\cat{A})$ partition $\stab{\cat{C}}$.  If the heart $\cat{A}$ has finite length and  $n$ simple objects then $U(\cat{A})\cong \U^n$ where $$\U = \{ z\in \C \ |\ \mathrm{Im}\, z \geq 0\} - \R_{\geq 0}.$$ 
In this case $U(\cat{A}) \cap \overline{U(\cat{B})}$ is non-empty and of codimension one when $\cat{B}$ is a simple left tilt of $\cat{A}$ \cite[Lemma 5.5]{Bridgeland:fk}, \ie when $\cat{B}$ is obtained from $\cat{A}$ by tilting in the sense of \cite{Happel:1996uq} at a torsion theory generated by a simple object.

We generalise this result as follows. Fix a heart $\cat{A}$ and assume that any heart obtained from it by a finite sequence of simple tilts 
has finite length and 
finitely many indecomposable objects. In particular each of these hearts therefore has 
finitely many simple objects and finitely many torsion theories.  In this case we show that $$U(\cat{A}) \cap \overline{U(\cat{B})} \neq \emptyset$$ if and only if $\cat{B}$ is the left tilt of $\cat{A}$ at some torsion theory. The codimension of the intersection is governed by the size of the torsion theory, lower codimension corresponding to a smaller torsion theory (see Corollary~\ref{intersection corollary} for a precise statement). Proposition~\ref{constraints} says that points in the boundary of $\overline{U(\cat{A})}$ correspond to limiting central charges of stability conditions in $U(\cat{A})$ for which the central charge of no semi-stable object vanishes. Furthermore, the union of subsets of stability conditions with hearts obtained from $\cat{A}$ by sequences of simple tilts forms a connected component of $\stab{\cat{C}}$ (Theorem~\ref{component thm}). Together these results yield a description of a component of $\stab{\cat{C}}$ in terms of the combinatorics of tilting in $\cat{C}$.

In brief, \S\ref{torsion theories}, \S\ref{tilting} and \S\ref{tilting at simples} provide the necessary background on the relation between $t$-structures, torsion theories and tilting in triangulated categories.  \S\ref{stability conditions} recalls the definition and basic properties of stability conditions. This material is standard and is included for expository purposes only. The main results, outlined above, are in \S\ref{main results}. They are illustrated in \S\ref{constructible} in which we compute (a component of) the space of stability conditions on the constructible derived category of $\P^1$ stratified by a point and its complement. (Algebraists may prefer to view this example as the derived category of representations of the bound quiver
$$
\xymatrix{\cdot \ar@/{}_{.5pc}/[r]_c& \cdot \ar@/{}_{.5pc}/[l]_v} \qquad vc=0
$$
and symplectic geometers as the subcategory of the derived asymptotic Fukaya category of $T^*\P^1$ generated by the zero section and a cotangent fibre.) This example is sufficiently simple to allow explicit computation and yet complicated enough to exhibit all the interesting features of the results. Finally \S\ref{coherent} discusses the contrasting case of coherent sheaves on $\P^1$, following the detailed treatment in \cite{MR2219846}. In this example the hearts do not have finitely many indecomposables  and we no longer obtain an entire component of the space of stability conditions by iterating simple tilts.

I'm grateful to Ivan Smith and Michael Butler for several helpful conversations and to Tom Bridgeland for his comments and for pointing me towards the example of coherent sheaves on $\P^1$ as a useful counterpoint. I am indebted to the referee who spotted several gaps and errors in the original version. 

\section{Tilting and stability conditions}
\label{theory}

We fix some notation. Let $\cat{C}$ be an additive category. We write $c\in \cat{C}$ to mean $c$ is an object of $\cat{C}$. We will use the term \defn{subcategory} to mean strict, full subcategory. When $S$ is a subcategory we write $\cat{S}^\perp$ for the subcategory on the objects 
$$
\{ c \in \cat{C} \ | \ \mor{s}{c} = 0 \ \forall s \in \cat{S} \}
$$
and similarly ${}^\perp\cat{S}$ for $\{ c \in \cat{C} \ | \ \mor{c}{s} = 0 \ \forall s \in \cat{S} \}$. When $\cat{A}$ and $\cat{B}$ are subcategories of $\cat{C}$ we write $\cat{A}\cap \cat{B}$ for the subcategory on objects which lie in both $\cat{A}$ and $\cat{B}$.

Suppose $\cat{C}$ is triangulated with shift functor $[1]$. Exact triangles in $\cat{C}$ will be denoted either by $a\to b\to c \to a[1]$ or by a diagram
$$
\xymatrix{
a \ar[rr] && b \ar[dl] \\
& c \ar@{-->}[ul] & 
}
$$
where the dotted arrow denotes a map $c \to a[1]$. We will always assume that $\cat{C}$ is essentially small so that isomorphism classes of objects form a set. Given sets $S_i$ of objects for $i\in I$  let $\langle S_i\ |\ i\in I \rangle$ denote the ext-closed subcategory generated by objects isomorphic to an element  in some $S_i$. We will use the same notation when the $S_i$ are subcategories of $\cat{C}$. 

\subsection{$t$-structures}
\label{torsion theories}

\begin{definition}
A \defn{$t$-structure} on a triangulated category $\cat{C}$ is a subcategory $\cat{D}\subset \cat{C}$ such that $\cat{D}[1] \subset \cat{D}$ and for each $c\in \cat{C}$ there is an exact triangle $d \to c \to d' \to d[1]$ with $d \in \cat{D}$ and $d'\in \cat{D}^\perp$. The subcategory $\cat{D} \cap \cat{D}^\perp[1]$ is abelian \cite[Th\'eor\`eme 1.3.6]{bbd} and is known as the \defn{heart} of the $t$-structure.
\end{definition}
 It is more common to define a $t$-structure to be a pair $(\cat{D}^{\leq 0}, \cat{D}^{\geq0})$ of sub-categories satisfying a short list of conditions, see \cite[\S1.3]{bbd}. This definition is equivalent to the one above if we put $\cat{D} = \cat{D}^{\leq 0}$ and $\cat{D}^\perp = \cat{D}^{\geq 0}[-1]$. The subcategory $\cat{D}$ is sometimes referred to as a $t$-category or an aisle. 
 
 It follows from the existence of the triangle that $\cat{D}$ is right admissible, \ie that there is a right adjoint $\tau^{\leq 0}$ to the inclusion $\cat{D} \hookrightarrow \cat{C}$ and that $\cat{D}^\perp$ is left admissible with left adjoint $\tau^{\geq 1}$ to the inclusion $\cat{D}^\perp \hookrightarrow \cat{C}$. These adjoints are referred to as truncation functors. The exact triangle associated to an object $c$ is unique (up to isomorphism) and can be written 
$$
\tau^{\leq 0}c \to c \to \tau^{\geq 1}c \to \tau^{\leq 0}c [1]
$$
where the first two maps come respectively from the counit and unit of the adjunctions. The truncations give rise to cohomological functors $$H^n = \tau^{\leq n}\tau^{\geq n}: \cat{C} \to \cat{A}$$ to the heart $\cat{A}$ where $\tau^{\leq n} c = \tau^{\leq 0} (c[n])[-n]$ and $\tau^{\geq n} c = \tau^{\geq 1} (c[n-1])[1-n]$.

A $t$-structure $\cat{D}$ is \defn{bounded} if any object of $\cat{C}$ lies in $\cat{D}[-n] \cap \cat{D}^\perp[n]$ for some $n\in \N$. In the sequel we will always assume that $t$-structures are bounded. This has two important consequences. Firstly, a bounded $t$-structure is completely determined by its heart $\cat{A}$; the $t$-structure is recovered as 
$$
\langle \cat{A}, \cat{A}[1] , \cat{A}[2] , \ldots \rangle.
$$
Secondly, the inclusion $\cat{A} \hookrightarrow \cat{C}$ induces an isomorphism $K(\cat{A}) \cong K(\cat{C})$ of Grothendieck groups. Closely related to this is the easy but important fact that if $\cat{A} \subset \cat{B}$ are hearts of (bounded) $t$-structures then $\cat{A}=\cat{B}$.

A $t$-structure is said to be \defn{faithful} if there is an equivalence $\cat{D}(\cat{A}) \simeq \cat{C}$
where $\cat{D}(\cat{A})$ is the bounded derived category of the heart $\cat{A}$. Note that, in general, there is not even a naturally defined functor $F: \cat{D}(\cat{A}) \to \cat{C}$ extending the inclusion $\cat{A} \to \cat{C}$. However, if such an $F$ does exist then it is an equivalence if and only if the Ext groups (in $\cat{C}$) between any two objects in $\cat{A}$ are generated in degree $1$ \cite[Chapter 5, Theorem 3.7.3]{gema}.  In the sequel, we do not assume that $t$-structures are faithful (although some of those we consider will be).

\subsection{Torsion theories and tilting}
\label{tilting}

\begin{definition}
A \defn{torsion theory} in an abelian category $\cat{A}$ is a subcategory $\cat{T}$ such that every $a\in \cat{A}$ fits into a short exact sequence
$$
0 \to t \to a \to f \to 0
$$
for some $t\in \cat{T}$ and $f\in \cat{T}^\perp$.  Objects of $\cat{T}$ are known as torsion objects and objects of $\cat{F}=\cat{T}^\perp$ as free objects; the motivating example is the subcategories of torsion and free abelian groups. 
\end{definition}
Torsion theories are more commonly defined as pairs $(\cat{T},\cat{F})$ of subcategories such that $\mor{t}{f}=0$ whenever $t\in \cat{T}$ and $f\in \cat{F}$ and every $a\in \cat{A}$ sits in a short exact sequence $0\to t \to a \to f \to 0$ with $t\in \cat{T}$ and $f\in \cat{F}$, see for example \cite[Definition 1.1]{MR2327478}. The definitions are equivalent: if $(\cat{T}, \cat{F})$ is a torsion theory then $\cat{F}= \cat{T}^\perp$.

  The short exact sequence $0\to t\to a \to f \to 0$ is unique up to isomorphism. The first term determines a right adjoint to the inclusion $\cat{T} \hookrightarrow \cat{A}$ and the last term a left adjoint to the inclusion $\cat{F} \hookrightarrow \cat{A}$. It follows that $\cat{T}$ is closed under factors, extensions and coproducts and that $\cat{F}$ is closed under subobjects, extensions and products. 

Torsion theories in $\cat{A}$ form a poset with the ordering $\cat{T} \leq \cat{T}' \iff \cat{T} \subset \cat{T}'$. The (bounded) $t$-structures on $\cat{C}$ also form a poset under the ordering $\cat{D} \leq \cat{D}' \iff \cat{D} \subset \cat{D}'$. 
\begin{proposition}
\label{HRS}
Let $\cat{C}$ be a triangulated category, and $\cat{D} \subset \cat{C}$ a $t$-structure with heart $\cat{A}$.  Then there is a canonical isomorphism between the poset of torsion theories in $\cat{A}$ and the interval in the poset of $t$-structures consisting of $t$-structures $\cat{E}$ with $\cat{D} \subset \cat{E} \subset \cat{D}[-1]$. 
\end{proposition}
\begin{proof}
Denote the cohomological functors associated to $\cat{D}$ by $H^i$. Given a torsion theory $\cat{T}$ in $\cat{A}$ define a subcategory
$$
\cat{E} = \langle \cat{D}, \cat{T}[-1] \rangle = \langle c \in \cat{C}\ |\ H^ic=0 \ \textrm{for}\  i>1 \ \textrm{and}\ H^1c \in \cat{T}\rangle.
$$
Happel, Reiten and Smal\o\ \cite[Proposition 2.1]{Happel:1996uq} show that this is a $t$-structure. It is clear that $\cat{D} \subset \cat{E} \subset \cat{D}[-1]$. 

Conversely, given a $t$-structure $\cat{E}$ with $\cat{D}\subset \cat{E} \subset \cat{D}[-1]$ we define $$\cat{T}=(\cat{E} \cap \cat{D}^\perp)[1] = \langle a\in\cat{A}\ |\ a=H^1e \ \textrm{for some} \ e\in \cat{E} \rangle.$$ 
Beligiannis and Reiten \cite[Theorem 3.1]{MR2327478} show that this is a torsion theory in $\cat{A}$. These constructions preserve the partial orders and are mutually inverse. 
\end{proof}

Let $\cat{A}$ be the heart of a $t$-structure $\cat{D}$. By the above a torsion theory $\cat{T}$ in $\cat{A}$ determines a new $t$-structure $\langle \cat{D}, \cat{T}[-1] \rangle$: we say this new $t$-structure is obtained from the original by \defn{left tilting at $\cat{T}$} and denote its heart by $L_\cat{T}\cat{A}$. Explicitly
$$
L_\cat{T}\cat{A} = \langle \cat{F} , \cat{T}[-1] \rangle \\
= \{ c \in \cat{C} \ | \ \textrm{$H^{0}c \in \cat{F}, H^1c\in \cat{T}$ and $H^ic=0$ for $i \neq 0,1$} \},
$$
where $H^i$ denotes the $i$th cohomology functor with respect to the initial $t$-structure and $\cat{F}=\cat{T}^\perp$. A torsion theory also determines another $t$-structure 
$$
\langle \cat{D}^\perp, \cat{F} \rangle^\perp = \langle c\in \cat{C}\ |\ H^ic=0 \ \textrm{for}\  i<0 \ \textrm{and}\ H^0c \in \cat{F}\rangle^\perp.
$$
by a `double dual' construction. This is the shift by $[1]$ of the other: we say it is obtained by \defn{right tilting} at $\cat{T}$ and denote the new heart by $R_\cat{T}\cat{A}$.  Intuitively, $L_\cat{T}\cat{A}$ is obtained from $\cat{A}$ by replacing $\cat{T}$ by its shift $\cat{T}[-1]$ and $R_\cat{T}\cat{A}$ is obtained by replacing $\cat{F}$ by $\cat{F}[1]$. Left and right tilting are inverse to one another: $\cat{F}$ is a torsion theory in $L_\cat{T}\cat{A}$ and right tilting with respect to this we recover the original heart $\cat{A}$. 
\begin{figure}[htbp]
\centerline {
\includegraphics[width=3.5in]{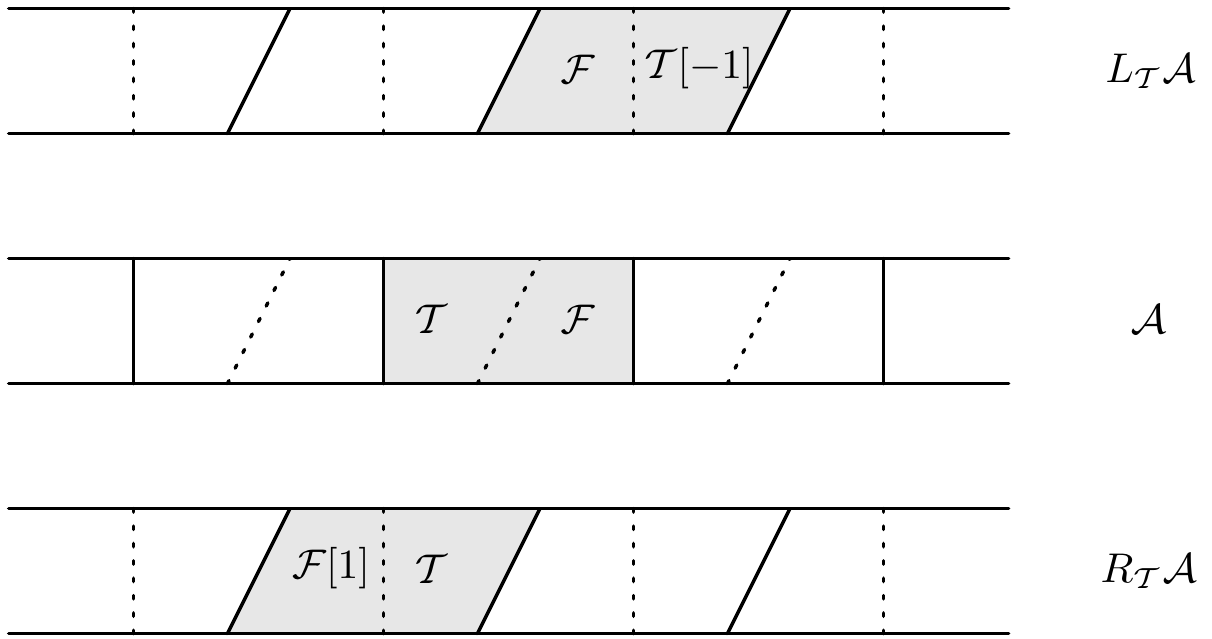}
}
\caption{Schematic diagram of tilting; indicated hearts shown shaded.}
\label{tilting picture}
\end{figure}

 \begin{lemma}
\label{simplicity criterion 1}
Let $\cat{T}$ be a torsion theory in the heart $\cat{A}$ of a $t$-structure. Then any simple object in $L_\cat{T}\cat{A}$ lies either in $\cat{F}=\cat{T}^\perp$ or in $\cat{T}[-1]$ and 
\begin{enumerate}
\item $f\in \cat{F}$ is simple in $L_\cat{T}\cat{A} \iff$ there are no exact triangles
$$
f'\to f\to f'' \to f'[1] \qquad \textrm{or} \qquad t'[-1] \to f\to f'\to t' 
$$
with $f',f''\in \cat{F}$ and $t'\in \cat{T}$ and all non-zero;
\item $t[-1] \in \cat{T}[-1]$ is simple in $L_\cat{T}\cat{A} \iff$ there are no exact triangles
$$
t'\to t\to t'' \to t'[1] \qquad \textrm{or} \qquad t'[-1]\to t[-1] \to f'\to t'
$$
 with $t',t''\in \cat{T}$ and $f'\in \cat{F}$ and all non-zero.
\end{enumerate}
\end{lemma}
\begin{proof}
Since $\cat{F}$ is a torsion theory in $L_\cat{T}\cat{A}$ with $\cat{F}^\perp = \cat{T}[-1]$ every object $b\in L_\cat{T}\cat{A}$ fits into a short exact sequence $0\to f \to b \to t[-1] \to 0$, and therefore can only be simple if it lies in either $\cat{F}$ or $\cat{T}[-1]$. The necessity of the conditions is evident. We  prove sufficiency only in the second case, the argument for the first case is similar. 

Suppose $t[-1]\in T[-1]$  is not simple in $L_\cat{T}\cat{A}$. Then there must be a proper subobject, and since $\cat{T}[-1]$ is the free part of a torsion theory in $L_\cat{T}\cat{A}$ this subobject must also be in $\cat{T}[-1]$, \ie we have a short exact sequence
$$
0\to t'[-1]\to t[-1] \to b\to 0
$$
 with $t'\in \cat{T}$ and $b\in L_\cat{T}\cat{A}$. There is a short exact sequence $0\to f \to b \to t''[-1]\to 0$ with $f\in \cat{F}$ and $t''\in\cat{T}$. If $t''\neq0$ then $t''$ is a proper quotient of $t$ and we can change our choice of $t'$ (by enlarging it) to obtain a short exact sequence
$$
0\to t'[-1]\to t[-1] \to t''[-1]\to 0
$$
in $L_\cat{T}\cat{A}$. On the other hand, if $t''=0$ then $b=f$ and we obtain a short exact sequence 
$$
0 \to t'[-1]\to t[-1]\to f\to 0
$$
in $L_\cat{T}\cat{A}$. These sequences yield the claimed exact triangles. 
\end{proof}
We leave the reader to formulate and prove the analogue for right tilts.

\subsection{Simple tilts}
\label{tilting at simples}

We will be particularly interested in the smallest changes that can be made in $t$-structures by tilting. 
Assume that the heart $\cat{A}$ has finite length and finitely many simple objects. To spell this out, $\cat{A}$ is both artinian and noetherian so that every object has a finite composition series. By the Jordan-H\"older theorem, the graded object associated to such a composition series is unique up to isomorphism. The simple objects form a basis for the Grothendieck group, which is isomorphic to $\Z^n$, where $n$ is the number of simple objects. 

Under this assumption each simple object $s\in \cat{A}$ determines two torsion theories, namely $\langle s \rangle$ and ${}^\perp\langle s \rangle$. These are respectively minimal and maximal (in the sense of the poset of torsion theories) non-trivial torsion theories in $\cat{A}$. To harmonise our notation with \cite{Bridgeland:fk} we use the shorthand $L_s = L_{\langle s \rangle}$ and $R_s = R_{{}^\perp\langle s \rangle}$. We will refer to $L_s\cat{A}$ and $R_s\cat{A}$ as {\em simple tilts} of $\cat{A}$.

For the remainder of this section we make the following stronger assumption.
\begin{assumption}
\label{assumption1}
The heart $\cat{A}$ and all hearts obtained from it by finite sequences of simple tilts have finite length and finitely many simple objects. (Note that they all have the same number of simple objects, since the classes of these form a basis for the Grothendieck group.)
\end{assumption}

\begin{lemma}
\label{simplicity criterion 2}
Suppose $s\in\cat{A}$ is simple. For each simple $a\neq s$ in $\cat{A}$ there are unique, up to isomorphism, simples $L_sa\in L_s\cat{A}$ and $R_sa\in R_s\cat{A}$ fitting into exact triangles $t[-1]\to a\to L_sa \to t$ with $t\in \langle s \rangle$ and $f\to R_sa\to a \to f[1]$ with $f\in \langle s \rangle$ respectively. Each simple in $L_s\cat{A}$, respectively $R_s\cat{A}$, is either $s[-1]$, respectively $s[1]$, or of the form $L_sa$, respectively $R_sa$, for a unique simple $a$ in $\cat{A}$. Moreover, $R_{s[-1]}L_sa\cong a$ and $L_sR_{s[-1]}b\cong b$.
\end{lemma}
\begin{proof}
The cases are similar so we deal only with the left tilt. It follows immediately from Lemma~\ref{simplicity criterion 1} that $s[-1]$ is simple in $L_s\cat{A}$, and that every other simple must lie in $\langle s \rangle^\perp$. 

Suppose that $a\in \langle s \rangle^\perp$ is simple in $\cat{A}$. If $a$ is {\em not} simple in $L_s\cat{A}$ then by Lemma~\ref{simplicity criterion 1} it has $s[-1]$ as a subobject. Quotienting by this, and repeating if necessary, we obtain $b\in L_s\cat{A}$ with $\mor{s[-1]}{b}=0$. We use the assumption that $L_s\cat{A}$ has finite length to show this process of quotienting must terminate.  Since $\langle s \rangle^\perp$ is a torsion theory in $L_s\cat{A}$ the quotient $b$ of $a$ is in $\langle s \rangle^\perp$ too, in particular $b\in \cat{A}$. By construction, the composition factors of $b$ in $\cat{A}$ are $a$ and a number (possibly zero) of copies of $s$. It follows that there is no exact triangle
$$
f' \to b\to f'' \to f'[1]
$$
with $f',f'' \in \langle s \rangle^\perp$ and both non-zero.  Therefore, by Lemma~\ref{simplicity criterion 1}, the object $b$ is simple in $L_s\cat{A}$. The class of $b$ in the Grothendieck group  is  $[a]+r[s]$ for some $r\geq 0$. It follows that there is a unique such simple quotient (otherwise the simples in  $L_s\cat{A}$ would not form a basis). Set  $L_sa=b$. Considering classes in the Grothendieck group again, we see that $a\neq a'$ implies $L_sa\neq L_sa'$. Thus, by counting, each simple in $L_s\cat{A}$ other than $s[-1]$ is of the form $L_sa$ for some simple $a\in \cat{A}$.  

The last assertion follows from noting that $R_{s[-1]}L_sa$ is the unique simple in $\cat{A}$ fitting into a triangle
$$
a' [-1]\to R_{s[-1]}L_sa\to L_sa \to a'
$$
with $a' \in \langle s\rangle$. Since $a$ also fits into the second term of this triangle, uniqueness tells us that $R_{s[-1]}L_sa\cong a$. The proof that $L_{s}R_{s[-1]}b\cong b$ is similar.
 \end{proof}
\begin{remark}
If we define $L_ss=s[-1]$ and $R_ss=s[1]$ then we obtain natural bijections between the simples in $\cat{A}$ and those in $L_s\cat{A}$ and $R_s\cat{A}$ respectively. There are no such natural bijections  for more general tilts. In the example we compute in \S\ref{constructible} it is possible to perform a sequence of  simple tilts beginning and ending at the same heart for which the composite of the above natural bijections is not the identity.
\end{remark}

The next lemma shows that tilting left by a torsion theory containing only finitely many indecomposables is equivalent to performing a sequence of simple left tilts. There is an obvious analogue for right tilts.
\begin{lemma}
\label{factorisation}
Suppose $a_0,\ldots,a_{N-1}$ is a finite sequence of objects in $\cat{A}$ such that $a_i\in \cat{A}$ is simple in $L_{a_{i-1}}\ldots L_{a_0}\cat{A}$. Then $L_{a_{N-1}}\ldots L_{a_0}\cat{A}=L_\cat{T}\cat{A}$ where $\cat{T} = \langle a_0,\ldots,a_{N-1}\rangle$. Furthermore, any left tilt at a torsion theory containing finitely many (isomorphism classes of) indecomposables is obtained in this way as a composition of simple left tilts.
\end{lemma}
\begin{proof}
Write  $\cat{A}_0$ for $\cat{A}$ and $\cat{A}_i$ for $L_{a_{i-1}}\ldots L_{a_0}\cat{A}$. Denote the $t$-structure with heart $\cat{A}$ by $\cat{D}_\cat{A}$ and so on. By Proposition~\ref{HRS}
$$
\cat{D}_{\cat{A}_N} = \langle \cat{D}_{\cat{A}_{N-1}} , a_{N-1}[-1] \rangle= \cdots
= \langle \cat{D}_{\cat{A}_0}, a_0[-1],\ldots,a_{N-1}[-1] \rangle = \langle \cat{D}_\cat{A} , \cat{T}[-1] \rangle
$$
where $\cat{T} = \langle a_0,\ldots,a_{N-1}\rangle$. Hence $\cat{A}_N=L_{a_{N-1}}\ldots L_{a_0}\cat{A}=L_\cat{T}\cat{A}$ as claimed. 

Now suppose $\cat{T}=\cat{T}_0$ is a (non-zero) torsion theory in $\cat{A}_0$ containing only finitely many indecomposables. Since $\cat{T}_0$ is closed under factors it must contain at least one simple object, $a_0$ say. We claim that $\cat{T}_1 = \cat{T}_0 \cap \langle a_0\rangle^\perp$ is a torsion theory in $\cat{A}_1$ with corresponding free theory $\cat{F}_1 = \langle \cat{F}_0, a_0[-1]\rangle$ where $\cat{F}_0=\cat{T}_0^\perp$. To see this, consider $b\in \cat{A}_1$. There is a short exact sequence
$$
0\to b' \to b \to b'' \to 0
$$
in $\cat{A}_1$ with $b' \in \langle a_0\rangle^\perp$ and $b'' \in \langle a_0[-1]\rangle$. Since $ \langle a_0\rangle^\perp \subset \cat{A}_0$  there is then a short exact sequence
$$
0 \to t \to b' \to f \to 0
$$
in $\cat{A}_0$ with $t\in \cat{T}_0$ and $f\in \cat{F}_0$. Furthermore, $t$ must be in $\langle a_0\rangle^\perp$ as $b'$ is, so that $t\in \cat{T}_1$. The three objects in the second short exact sequence all lie in $\cat{A}_1$ so that it is, in fact, short exact in $\cat{A}_1$ too. Hence $t$ is a subobject of $b$ in $\cat{A}_1$ and the quotient is an extension of $b''$ by $f$ and so lies in $\cat{F}_1$. Since $\cat{F}_1 \subset \cat{T}_1^\perp$ the claim follows. 

Note that $\cat{T}_1 \subset \cat{T}_0$ (although we do not claim that $\cat{T}_1$ is a torsion theory in $\cat{A}_0$).  Repeating the argument we obtain a nested sequence of full subcategories in $\cat{T}_0$. Since $\cat{T}_0$ has only finitely many indecomposables, and at each stage we remove at least one indecomposable, namely the simple $a_i \in \cat{A}_i$, we must eventually reach $\cat{T}_N=0$. 

By the first part $\langle a_0,\ldots,a_{N-1} \rangle$ is a torsion theory in $\cat{A}_0$ with 
$$
\cat{A}_N=L_{a_{N-1}}\ldots L_{a_0}\cat{A}_0=L_{\langle a_0,\ldots,a_{N-1}\rangle}\cat{A}_0.
$$
Clearly $\langle a_0,\ldots,a_{N-1}\rangle \subset \cat{T}_0$. In fact they must be equal, so that $\cat{A}_N=L_{\cat{T}_0}\cat{A}_0=L_\cat{T}\cat{A}$ as required. If not choose $b\in \cat{T}_0$ but not in $\langle a_0,\ldots,a_{N-1}\rangle$. We may assume that $b \in \langle a_0,\ldots,a_{N-1}\rangle ^\perp$. If it is not then there is a monomorphism $a \rightarrowtail b$ for some $a\in \langle a_0,\ldots,a_{N-1}\rangle $ and we may replace $b$ by the cokernel, which must also lie in $\cat{T}_0$ but not $\langle a_0,\ldots,a_{N-1}\rangle$. Iterating, and using the fact that $\cat{A}_0$ has finite length to ensure termination, we may indeed assume $b \in \langle a_0.\ldots,a_{N-1}\rangle^\perp$. We may also assume $b$ is indecomposable; if not replace it with a summand, at least one of which cannot be in $\langle a_0,\ldots,a_{N-1}\rangle$. Hence there is an indecomposable
$$
b \in \cat{T}_0 \cap \langle a_0,\ldots,a_{N-1}\rangle^\perp = \cat{T}_N
$$
which contradicts $\cat{T}_N=0$.
\end{proof}

\begin{remark}
The actions of automorphisms of $\cat{C}$ commute with simple tilts. More precisely, if $\alpha: \cat{C} \to \cat{C}$ is an automorphism, \ie a triangulated equivalence, then 
$$
\alpha( L_s\cat{A} ) = L_{\alpha(s)}\alpha(\cat{A}) \quad \textrm{and} \quad \alpha( R_s\cat{A} ) = R_{\alpha(s)}\alpha(\cat{A}),
$$
see \cite[Lemma 8.2]{MR2373143}. In the example in \S\ref{constructible} the triangulated category $\cat{C}$ will be equipped with a duality, a triangulated anti-equivalence $D: \cat{C}^{op}\to \cat{C}$ such that $D^2\cong \id$. Thus every $t$-structure will have an associated dual $t$-structure, given a torsion theory $\cat{T}$  in the heart $\cat{A}$ of a $t$-structure there will be a dual torsion theory $D\cat{T}^\perp$ in the heart $D\cat{A}$ of this dual $t$-structure and left tilting $\cat{A}$ with respect to $\cat{T}$ is dual to right tilting $D\cat{A}$ with respect to $D\cat{T}^\perp$. In particular if $s$ is a simple object of $\cat{A}$ then $D(L_s\cat{A}) = R_{Ds}D\cat{A}$.
\end{remark}

\subsection{Stability conditions}
\label{stability conditions}

Let $\cat{C}$ be a triangulated category and $K(\cat{C})$ be its Grothendieck group. A \defn{stability condition} $(\mathcal{Z},\mathcal{P})$ on $\cat{C}$ \cite[Definition 1.1]{MR2373143} consists of a group homomorphism $\mathcal{Z} : K(\cat{C}) \to \C$ and full additive subcategories $\mathcal{P}(\varphi)$ of $\cat{C}$ for each $\varphi\in \R$ satisfying
\begin{enumerate}
\item if $c\in \mathcal{P}(\varphi)$ then $\mathcal{Z}(c) = m(c)\exp(i\pi \varphi)$ where $m(c) \in \R_{>0}$;
\item $\mathcal{P}(\varphi+1) = \mathcal{P}(\varphi)[1]$ for each $\varphi\in \R$;
\item if $c\in \mathcal{P}(\varphi)$ and $c' \in \mathcal{P}(\varphi')$ with $\varphi > \varphi'$ then $\mor{c}{c'}=0$;
\item\label{HN filtration} for each nonzero object $c\in \cat{C}$ there is a finite collection of triangles
$$
\xymatrix{
0=c_0 \ar[rr] && c_1 \ar[r] \ar[dl] & \cdots \ar[r] & c_{n-1} \ar[rr] && c_n=c \ar[dl]\\
& b_1\ar@{-->}[ul] &&&& b_n \ar@{-->}[ul]
}
$$
with $b_j \in \mathcal{P}(\varphi_j)$ where $\varphi_1 > \cdots > \varphi_n$.
\end{enumerate}
The homomorphism $\mathcal{Z}$ is known as the \defn{central charge} and the objects of $\mathcal{P}(\varphi)$ are said to be \defn{semi-stable of phase $\varphi$}. The objects $b_j$ are known as the \defn{semi-stable factors} of $c$. We define  $\varphi^+(c)=\varphi_1$ and $\varphi^-(c)=\varphi_n$. The \defn{mass}  of $c$ is defined to be $m(c)=\sum_{i=1}^n m(b_i)$. 

A stability condition $(\mathcal{Z},\mathcal{P})$ determines a bounded $t$-structure on $\cat{C}$ with heart the extension-closed subcategory generated by semi-stables with phases in $(0,1]$. Conversely, if we are given a bounded $t$-structure on $\cat{C}$ together with a stability function on the heart with the Harder--Narasimhan property --- the abelian analogue of property~\ref{HN filtration} above --- then this determines a stability condition on $\cat{C}$ \cite[Proposition 5.3]{MR2373143}.

We say a stability condition is \defn{locally-finite} if we can find $\epsilon >0$ such that the quasi-abelian category $\mathcal{P}(t-\epsilon, t+\epsilon)$, generated by semi-stables with phases in $(t-\epsilon, t+\epsilon)$, has finite length (see \cite[Definition 5.7]{MR2373143}). The set of locally-finite stability conditions can be topologised so that it is a, possibly infinite-dimensional, complex manifold, which we denote $\stab{\cat{C}}$ \cite[Theorem 1.2]{MR2373143}. There is a linear subspace $V \subset \mor{K(\cat{C})}{\C}$ such that the projection $(\mathcal{Z}, \mathcal{P}) \mapsto \mathcal{Z}$ is a local homeomorphism on $V$. The slogan is that deformations of the central charge lift to deformations of the stability condition. The topology arises from the (generalised) metric
$$
d(\sigma,\tau) = \sup_{0\neq c \in \cat{C}} \max\left( | \varphi_\sigma^-(c) - \varphi_\tau^-(c)| , | \varphi_\sigma^+(c) - \varphi_\tau^+(c)|, \left| \log \frac{m_\sigma(c)}{m_\tau(c)}\right| \right) 
$$
which takes values in $[0,\infty]$.

The group ${\rm Aut}(\cat{C})$ of automorphisms acts continuously on the space $\stab{\cat{C}}$ of stability conditions with an automorphism $\alpha$ acting by
$$
(\mathcal{Z}, \mathcal{P}) \mapsto \left(\mathcal{Z}\circ \alpha^{-1}, \alpha(\mathcal{P})\right).
$$
\label{c action} There is also a free action of $\C$ on any space of stability conditions given by 
$$
\lambda \cdot \left(\mathcal{Z},\mathcal{P}(\varphi) \right) = \left( \exp(i\pi\lambda)\mathcal{Z} , \mathcal{P}(\varphi + \mathrm{Re} \lambda) \right).
$$
This action preserves the semi-stable objects, but changes their phases and masses. The action of $1\in \C$ corresponds to the action of $[1]$. 

A stability condition determines a family of torsion theories in the heart of any $t$-structure.
\begin{lemma}
\label{stab-fam-tor}
Let $\cat{A}$ be the heart of some $t$-structure on $\cat{C}$ and $\sigma$ a stability condition. Then for each $t\in\R$ there is a torsion theory $\cat{T}_t = \langle  a \in \cat{A} \ |\ \varphi_\sigma^-(a) > t \rangle$ in $\cat{A}$.
\end{lemma}
\begin{proof}
Set $\cat{F}_t = \langle a \in \cat{A}\ |\ \varphi_\sigma^+(a) \leq t \rangle$. Each $a\in \cat{A}$ has a filtration in terms of semi-stable objects with respect to $\sigma$. Combining the semi-stable factors of phase $>t$ into $a'$ we obtain a short exact sequence
$$
0\to a' \to a\to a'' \to 0
$$
 in $\cat{A}$ with $\varphi^-_\sigma(a') > t$ and $\varphi^+_\sigma(a'')\leq t$, \ie with $a'\in \cat{T}_t$ and $a''\in \cat{F}_t$. Since there are no non-zero morphisms from a semi-stable of phase $>t$ to one of phase $\leq t$ we see that $\cat{F}_t \subset  \cat{T}_t^\perp$. 
\end{proof}

\subsection{Tiling by tilting}
\label{main results}

Our strategy for understanding the space of stability conditions is based on \cite[\S 5]{Bridgeland:fk}. Let $U(\cat{A}) \subset \stab{\cat{C}}$ be the subset of locally-finite stability conditions with heart $\cat{A}$. 
Clearly these subsets partition the space of stability conditions. If $\cat{A}$ has finite length and $n$ simple objects then a stability condition with heart $\cat{A}$ is determined by any non-zero choice of central charge with phase in $(0,1]$ for each of the $n$ simples. Hence the subset $U(\cat{A})$ is isomorphic to $\U^n$ where
$$
\U = \{ r\exp(i\pi\varphi) \ :\ r \in \R_{>0}, 0< \varphi \leq 1 \}.
$$
We can think of the $U(\cat{A})$ for such hearts $\cat{A}$ as tiles covering part of the space of stability conditions. The next result shows that performing a simple tilt corresponds to moving to an adjacent tile. 

\begin{proposition}[{See \cite[Lemma 5.5]{Bridgeland:fk}}]
\label{tiling}
Let $\cat{A} \subset \cat{C}$ be the heart of a bounded $t$-structure and suppose $\cat{A}$ has finite length and $n$ simple objects. Then, if $S\subset U(\cat{A})$ is the codimension $1$ subset for which the simple $s$ has phase $1$ and all other simples in $\cat{A}$ have phases in $(0,1)$, 
$$
U(\cat{A}) \cap \overline{U(\cat{B})} = S \iff \cat{B} = L_s\cat{A}.
$$
\end{proposition}

When $\cat{B}=L_s\cat{A}$ the codimension one faces of $U(\cat{A})$ and $\overline{U(\cat{B})}$ are glued using a linear map expressing the change of basis in $K(\cat{C})$ from a basis of simples in $\cat{A}$ to a basis of simples in $\cat{B}$. It follows from Lemma~\ref{simplicity criterion 2} that this gluing map
$$
\overline{U(\cat{B})} \supset \mathbb{U}^{n-1} \times \R_{>0} \longrightarrow  \mathbb{U}^{n-1} \times \R_{<0} \subset U(\cat{A})
$$
where $\mathbb{U} = \{ z\in \C\ |\ \mathrm{Im}\, z >0 \}$ is the strict upper half-plane, is given by a matrix of the form
$$
\left(
\begin{array}{ccc|c}
1  &   & 0 & -m_1 \\
  & \ddots  && \vdots  \\
 0 &   & 1 &  -m_{n-1} \\
 \hline
 0&\cdots&0& -1
\end{array}
\right)
$$
where $m_i \in \N$ for $i=1,\ldots, n-1$.

Given Proposition~\ref{tiling} it is natural to ask: Can we obtain an entire component of the space of stability conditions by iterating simple tilts? Do the tiles $U(\cat{A})$ and $U(L_\cat{T}\cat{A})$ share a boundary for a general torsion theory $\cat{T}$? In general the answers are both negative, but in this section we show that the answers are affirmative in the presence of some (quite strong) assumptions.

Fix a heart $\cat{A}$ and let $\{\cat{A}_i \ | \ i\in I\}$ be the set of hearts obtained from $\cat{A}$ by a finite sequence of simple tilts. In order that we can continue to perform simple tilts indefinitely we assume that Assumption~\ref{assumption1} holds, \ie that each heart $\cat{A}_i$ has finite length and finitely many simple objects. 

The subset $U(\cat{A}) \subset \stab{\cat{C}}$ is naturally stratified with a stratum of codimension $k$ corresponding to a collection of $k$ simples $\{s_1,\ldots,s_k\}$ in $\cat{A}$. The corresponding stratum $S_{\{s_1,\ldots,s_k\}}$ consists of those $\sigma \in U(\cat{A})$ for which the $s_i$ have phase $1$ and all other simples have phases in $(0,1)$. The stratum is  isomorphic to $\mathbb{U}^{n-k}\times \R_{<0}^k$.

Suppose $s\in \cat{A}$ is simple and $\sigma\in \overline{U(\cat{A})}$. We say $L_s$ is a \defn{simple left tilt about $\sigma$} if $\varphi_\sigma(s)=1$ and a \defn{simple right tilt about $\sigma$} if $\varphi_\sigma(s)=0$. Note that the subset of stability conditions for which $s$ is semi-stable is closed \cite[Proposition 8.1]{MR2373143}, in particular $s$ is semi-stable for any $\sigma\in\overline{U(\cat{A})}$ so that these definitions make sense. 

\begin{lemma}
If $L_s$ is a simple left tilt about $\sigma\in \overline{U(\cat{A})}$ then $\sigma \in \overline{U(L_s\cat{A})}$. There is an obvious analogue for simple right tilts.
\end{lemma}
\begin{proof}
Since $\sigma \in \overline{U(\cat{A})}$ and $\varphi_\sigma(s)=1$ the stability condition $\sigma$ is contained in the closure of the stratum $S_{\{s\}}$ in $U(\cat{A})$, and hence $\sigma \in \overline{U(L_a\cat{A})}$ too by Proposition~\ref{tiling}. 
\end{proof}
\begin{corollary}
\label{intersection corollary}
Suppose $\cat{T} \subset \langle s_1, \ldots,s_k\rangle$ is a torsion theory in $\cat{A}$ containing only finitely many indecomposables. Then $$S_{\{s_1,\ldots,s_k\}} \subset U(\cat{A}) \cap \overline{U(L_\cat{T}\cat{A})},$$
and in particular $\dim_{\R} U(\cat{A}) \cap \overline{U(L_\cat{T}\cat{A})} \geq k$.
\end{corollary}
\begin{proof}
By Lemma~\ref{factorisation} we can write $\cat{T}=\langle a_0,\ldots,a_{N-1}\rangle$ where $a_i\in \cat{A}$ is simple in $L_{a_{i-1}}\ldots L_{a_0}\cat{A}$ and left tilting at $\cat{T}$ is equivalent to tilting left successively at the simples $a_0,\ldots,a_{N-2}$ and $a_{N-1}$. Each $a_i$ is in $\langle s_1, \ldots,s_k\rangle$ so that $L_{a_i}$ is a simple left tilt about $\sigma$ for any $\sigma\in S_{\{s_1,\ldots,s_k\}}$. 
\end{proof}

\begin{lemma}
\label{tilting about stratum}
Suppose $\cat{B}$ is obtained from $\cat{A}$ by a finite sequence of left tilts about  $\sigma\in \overline{U(\cat{A})}$. Then there is a torsion theory $\cat{T}$ in $\cat{A}$ such that $\cat{B} = L_{\cat{T}}\cat{A}$. 
\end{lemma}
\begin{proof}
The proof is by induction on the length of the sequence of left tilts about $\sigma$. It is clear for a sequence of length $1$. Suppose it is true for  all sequences of length $m$ and that $\cat{B}$ is obtained from $\cat{A}$ by a sequence of $m+1$ tilts at simples about $\sigma$. Let $\cat{B}'$ be the heart obtained after the first $m$ left tilts. By induction there is a torsion theory $\cat{T}'$ in $\cat{A}$ such that $\cat{B}' = L_{\cat{T}'}\cat{A}$. The heart $\cat{B}$ is obtained from $\cat{B}'$ by a simple left tilt left at some $b\in \cat{B}'$. By Lemma~\ref{simplicity criterion 1}  either $b \in \cat{T}'[-1]$ or $b\in {\cat{T}'}^\perp$. In the latter case $b \in \cat{A}$ and the result follows from Proposition~\ref{HRS}.

It remains to show that $b \not \in  \cat{T}'[-1]$. Suppose it is, so that $b=t[-1]$ for some $t\in \cat{T} \subset \cat{A}$. Then for $\tau \in U(\cat{A})$ we have
$0< \varphi_\tau^-(t) \leq \varphi_\tau^+(t) \leq 1$
and, as $\varphi^\pm$ is continuous \cite[Proposition 8.1]{MR2373143}, and $\sigma\in\overline{U(\cat{A})}$, $$0\leq \varphi_\sigma^-(t) \leq \varphi_\sigma^+(t) \leq 1.$$ On the other hand, $b$ is simple in $\cat{B}'$ hence semi-stable in the closure $\overline{U(\cat{B}')}$. We can left tilt at $b$ about $\sigma$ so $1=\varphi_\sigma(b)=\varphi_\sigma(t)-1$ which is a contradiction.
\end{proof}
\begin{remark}
\label{nested remark}
The statement of the lemma is false for more general sequences of simple left (or right) tilts; for instance the heart $L_{s[-1]}L_s\cat{A}$ is not a left tilt of $\cat{A}$ at any torsion theory. A sequence of left tilts about $\sigma$ corresponds to a strictly increasing nested sequence of torsion theories in $\cat{A}$.
\end{remark}
In order to obtain the next result, a converse to Corollary~\ref{intersection corollary}, we need an additional assumption, whose purpose is to ensure that we cannot tilt  about a stability condition indefinitely.
\begin{assumption}
\label{assumption2}
Each heart $\cat{A}_i$ has 
finite length and
only finitely many indecomposables. 
\end{assumption}
\begin{remark}
Assumption~\ref{assumption2} implies that each $\cat{A}_i$ has 
finitely many simples and finitely many torsion theories. 
This follows because simple objects are indecomposable and because a torsion theory is completely determined by the indecomposables it contains. 
In particular Assumption~\ref{assumption2} implies Assumption~\ref{assumption1}.
\end{remark}

\begin{corollary}
\label{intersection corollary converse}
Suppose Assumption~\ref{assumption2} holds, that $s_1,\ldots,s_k$ are simple objects in $\cat{A}$ and $S_{\{s_1,\ldots,s_k\}} \subset U(\cat{A}) \cap  \overline{U(\cat{B})}$. Then $\cat{B}=L_\cat{T}\cat{A}$ for some $\cat{T}\subset \langle s_1,\ldots,s_k\rangle$.
\end{corollary}
\begin{proof}
Choose $\sigma \in S_{\{s_1,\ldots,s_k\}}$. If $\varphi_\sigma(b) \in (0,1]$ for every simple $b\in \cat{B}$ then $\sigma \in U(\cat{B})$ which implies $\cat{B}=\cat{A}$ and the result is immediate. Assume then that there is a simple $b\in \cat{B}$ with $\varphi_\sigma(b) = 0$. Hence $\sigma \in \overline{U(R_{b}\cat{B})}$ too. We repeat this step until we can no longer right tilt about $\sigma$. The process must terminate: by the analogue of Lemma~\ref{tilting about stratum} for right tilts  any such sequence is equivalent to a single right tilt at some torsion theory and by Remark~\ref{nested remark} longer sequences correspond to strictly larger  torsion theories. Since, by Assumption~\ref{assumption2}, there are only finitely many torsion theories we cannot continue indefinitely. When we can no longer right tilt about $\sigma$ then we have reached $\cat{A}$. Hence $\cat{B} = L_\cat{T}\cat{A}$ for some $\cat{T}$. By construction $\varphi_\sigma(t) =1$ for $t\in \cat{T}$ so $\cat{T} \subset \langle s_1, \ldots,s_k\rangle$ as required.
\end{proof}

Naively, one might hope that the closure of $U(\cat{A})$ is obtained by allowing the central charges of arbitrary subsets of simples to tend to real values, \ie that
$$
\overline{U(\cat{A})} \cong \{ z\in \C \ |\ \mathrm{Im}\, z \geq 0\ \textrm{and} \ z\neq 0\}^n
$$
stratified in the obvious fashion. However, there are constraints on the ways in which central charges of simples can degenerate to real values. These arise because the set of stability conditions for which an object is semi-stable is closed \cite[Proposition 8.1]{MR2373143} and the central charge of a semi-stable must be non-zero. Thus, for example, if $0 \to s \to a \to s'\to 0$ is short exact (and not split) where $s,s'$ are distinct simples then $a$ becomes semi-stable as $\mathcal{Z}(s) \to \R_{>0}$ and $\mathcal{Z}(s') \to \R_{<0}$. Hence degenerations to the boundary of $U(\cat{A})$ with $\mathcal{Z}(s) \to \R_{>0}$ and 
$$
\label{naive}
\mathcal{Z}(a) = \mathcal{Z}(s) + \mathcal{Z}(s') \to 0
$$
are forbidden. These constraints mean that certain hyperplanes are excised from the naively expected strata of codimension $\geq 2$. (There are no constraints of this kind when a set of central charges of simples all become either positive or negative reals.)

\begin{proposition}
\label{constraints}
Under Assumption~\ref{assumption2}, a stability condition $(\mathcal{Z},\mathcal{P})$ in the boundary of $\overline{U(\cat{A})}$ is determined by allowing the central charge of a stability condition in the interior of $U(\cat{A})$ to degenerate in such a way that the charges of a non-empty set of simples become real. The only constraint on this degeneration is that there is no $a\in \cat{A}$ with $\mathcal{Z}(a)=0$ which is semi-stable for each of a sequence of stability conditions in $U(\cat{A})$ with central charges $\mathcal{Z}_i\to \mathcal{Z}$. \end{proposition}
\begin{proof}
It follows from \cite[Lemma 5.2]{Bridgeland:fk} that any stability condition in the boundary of $\overline{U(\cat{A})}$ corresponds to such a degeneration of the central charge. It remains only to prove that there are no other constraints.

Let $\mathcal{Z}: K(\cat{C}) \to \C$ be the limit of a sequence of central charges of stability conditions in $U(\cat{A})$. Suppose that $\mathcal{Z}(s_1), \ldots, \mathcal{Z}(s_k) \in \R-\{0\}$ where $s_1,\ldots,s_k$ are simples in $\cat{A}$, that $\mathrm{Im}\, \mathcal{Z}(s) >0$ for any other simple $s$ in $\cat{A}$ and that $\mathcal{Z}(a) \neq 0$ for any $a$ which is semi-stable for a sequence of stability conditions in $U(\cat{A})$ whose central charges limit to $\mathcal{Z}$. 

Start at $\cat{A}$ and repeatedly right tilt at simples $s$ for which $\mathcal{\mathcal{Z}}(s)>0$. By the same argument as in the proof of Corollary~\ref{intersection corollary converse} this process terminates at a heart $\cat{B} = R_\cat{T}\cat{A}$ where 
$\cat{T}$ is a torsion theory in $\cat{A}$. Then $\mathcal{Z}$ is the limit of central charges of stability conditions in $U(\cat{B})$.  This implies $\mathrm{Im}\, \mathcal{Z}(b) \geq 0$ for any $b\in \cat{B}$. If in addition $b$ is simple then we must have $\mathcal{Z}(b) \in \U \cup\{0\}$, otherwise we could right tilt at $b$. It follows that $\mathcal{Z}(b) \in \U \cup\{0\}$ for any $b\in \cat{B}$.

Recall that $\cat{B} = R_\cat{T}\cat{A} = \langle \cat{T}, \cat{F}[1] \rangle$ where $\cat{F} = \cat{T}^\perp$. It follows immediately that $\mathcal{Z}(t) \in \U \cup\{0\}$ for any $t\in \cat{T}$ and $\mathcal{Z}(f[1])=-\mathcal{Z}(f) \in \U \cup \{0\}$ for any $f\in F$. On the other hand $\mathrm{Im}\, \mathcal{Z}(f) \geq 0$ because $f\in \cat{A}$ so $\mathcal{Z}(f) \in \R_{\geq 0}$. 

We now show $\mathcal{Z}(b)\in \U$ for any $b\in \cat{B}$. It then follows that $\mathcal{Z}$ is the central charge of a stability condition in $U(\cat{B})$ and that this stability condition is in the boundary of $\overline{U(\cat{A})}$, \ie this degeneration of the central charge corresponds to a degeneration of stability conditions as claimed. 

First suppose $b\in \cat{T}$ and $\mathcal{Z}(b)=0$. Consider a short exact sequence
$$
0\to a\to b\to a' \to 0
$$
in $\cat{A}$. Since $b$ is torsion so is $a'$. It follows that $\mathcal{Z}(a')\in \R_{\leq 0}$ and $\mathcal{Z}(a)\in \R_{\geq 0}$. In fact we must have $\mathcal{Z}(a')=0$ for some quotient $a'$. If not, then $\mathcal{Z}(a')<0$ and $\mathcal{Z}(a)>0$ for all such short exact sequences and this implies that $b$ is semi-stable for all stability conditions in $U(\cat{A})$ whose central charges are close to $\mathcal{Z}$. By assumption $\mathcal{Z}(b)\neq 0$ in this situation, which is a contradiction. Therefore, we can choose a quotient $a'$ with $\mathcal{Z}(a')=0$ and repeat the argument. After finitely many steps we obtain a simple quotient $s$ with $\mathcal{Z}(s)=0$ which again contradicts the assumptions on $\mathcal{Z}$. We conclude that $\mathcal{Z}(b) \neq 0$.

The second case is when $b\in \cat{F}[1]$ and $\mathcal{Z}(b)=0$. A similar argument, involving passing to subobjects rather than quotients, leads to a contradiction. Hence $\mathcal{Z}(b)\neq 0$ when $b$ is in $\cat{T}$ or in $\cat{F}[1]$. The result follows. 
\end{proof}

This gives us a good understanding of the closure $\overline{U(\cat{A})}$. It also follows from the proof that $\overline{U(\cat{A})} \subset \bigcup_{i\in I} U(\cat{A}_i)$ and hence that
$$
\bigcup_{i\in I} \overline{U(\cat{A}_i)} = \bigcup_{i\in I} U(\cat{A}_i).
$$
\begin{theorem}
\label{component thm}
Under Assumption~\ref{assumption2} the union $\bigcup_{i\in I} U(\cat{A}_i)$ is a component of the space of stability conditions.
\end{theorem}
\begin{proof}
We show that $\bigcup_{i\in I} U(\cat{A}_i)$ is both open and closed as a subset of the space of stability conditions. Let $\sigma\in U(\cat{A}_i)$ and consider the ball $B(\sigma,1/2)$. Suppose $B(\sigma,1/2) \cap U(\cat{B})\neq \emptyset$, say $\rho$ is a stability condition in the intersection. Consider the torsion theory (recall Lemma~\ref{stab-fam-tor}) 
$$
\cat{T} = \langle b\in \cat{B} \ | \ \varphi_\rho^-(b) >  1/2\rangle
$$
in $\cat{B}$ with free theory $\cat{F} = \langle b\in \cat{B} \ | \ \varphi_\rho^+(b)\leq 1/2\rangle$. Let $\cat{B}'=L_\cat{T}\cat{B} = \langle \cat{F}, \cat{T}[-1]\rangle$. It is clear that  $\varphi^\pm_\rho(b')\in (-1/2,1/2]$ for any $0\neq b'\in \cat{B}'$. Since $d(\rho,\sigma)<1/2$ we have
$$
-1 < \varphi^-_\sigma(b') \leq \varphi^+_\sigma(b') \leq 1
$$
for any $0\neq b'\in\cat{B}'$. Hence $\cat{B}' \subset \cat{D}_\cat{A}[-1]\cap\cat{D}_\cat{A}^\perp[1]$, or equivalently $\cat{D}_\cat{A} \subset \cat{D}_{\cat{B}'} \subset \cat{D}_\cat{A}[-1]$, and there is a torsion theory $\cat{T}'$ in $\cat{A}$ for which $\cat{B}' = L_{\cat{T}'}\cat{A}$. Therefore
$$
\cat{B} = R_{\cat{T}[-1]}L_{\cat{T}'}\cat{A}
$$
is obtained from $\cat{A}$ by tilting twice, first left and then right. (Of course one or both of these tilts may be trivial.) Under Assumption~\ref{assumption2} it follows that $\cat{B}=\cat{A}_i$ for some $i$. So 
$$
B(\sigma,1/2) \subset \bigcup_{i\in I} U(\cat{A}_i)
$$
and in particular $\bigcup_{i\in I} U(\cat{A}_i)$ is open. Furthermore if $\tau \in \overline{\bigcup_{i\in I} U(\cat{A}_i) }$ then  $\tau$ is in $B(\sigma,1/2)$ for some $\sigma$. By the above, it is in $\bigcup_{i\in I} U(\cat{A}_i)$ which is therefore closed.
\end{proof}
\begin{remark}
Assumption~\ref{assumption2} guarantees that the tiling of the component containing stability conditions with heart $\cat{A}$ by the subsets $U(\cat{A}_i)$ is locally-finite.
\end{remark}

The results of this section depend upon Assumption~\ref{assumption1} and, for the most part, on the stronger  Assumption~\ref{assumption2}. We make some remarks on two special situations in which the property of having finitely many indecomposables, or the weaker property of having finite length and finitely many simples are inherited by a tilted heart. We will see examples of each in the next section. For the remainder of this section we assume that $\cat{C}$ is $k$-linear for some field $k$ and of finite type, \ie $\bigoplus_i \Ext{\cat{C}}{i}{c}{c'}$ is finite dimensional for any objects $c$ and $c'$. 

\label{spherical remarks}
The first special situation arises when the simple $s$ is a \defn{spherical object}. This means that there is some $d\in \N$, the \defn{dimension} of the spherical object $s$, such that $\Ext{\cat{C}}{i}{s}{s}\cong k$ if $i=0$ or $d$ and vanishes otherwise, and that for any $c\in \cat{C}$ the pairing 
$$
\Ext{\cat{C}}{i}{s}{c}\otimes \Ext{\cat{C}}{d-i}{c}{s} \to \Ext{\cat{C}}{d}{s}{s} \cong k
$$
is perfect. A spherical object defines an automorphism, often called a twist, $\Phi_s$ of $\cat{C}$ characterised by the triangle
\begin{equation}
\label{twist formula}
\Ext{\cat{C}}{*}{s}{c}\otimes s \to c\to \Phi_s(c) \to \Ext{\cat{C}}{*}{s}{c}\otimes s[1],
\end{equation}
where the first map is evaluation, see \cite{MR1831820}. Suppose that $d\neq 1$ and that for $i\neq 1$
 $$
\Ext{\cat{C}}{i}{s}{s'}=0
$$
for any simple $s'\neq s$ in $\cat{A}$. Then by considering the long exact sequence obtained from (\ref{twist formula}) upon applying $\Ext{\cat{C}}{i}{s}{-}$ we see that $\Ext{\cat{C}}{*}{s}{\Phi_s(s')}=0$ for $i=0$ and $1$. Hence, by Lemmas~\ref{simplicity criterion 1} and~\ref{simplicity criterion 2}, we have $\Phi_s(s')=L_ss' $. Furthermore $\Phi_s(s) = s[1-d]$ so that when $d=2$ the tilted heart $L_s\cat{A}$ is the image of $\cat{A}$ under the automorphism $\Phi_s$. Thus it inherits all the properties of $\cat{A}$, in particular if $\cat{A}$ has finitely many indecomposables then so does $L_s\cat{A}$. Clearly we can repeat this construction starting from the tilted heart with $s[-1]$ in place of $s$.  

\label{excellent remarks}
The second special situation arises when we have an \defn{excellent collection} in the sense of \cite[\S 3]{MR2142382} (where the full details of what follows can be found). This is a full, strong exceptional collection $(e_0, \ldots, e_{n-1})$ of objects in $\cat{C}$ with an additional property that ensures that all mutations of the collection are also excellent. It follows that the braid group acts on the set of excellent collections of length $n$ in $\cat{C}$ by mutations. Each excellent collection in $\cat{C}$ determines a faithful $t$-structure whose heart we denote $\cat{A}(e_0, \ldots, e_{n-1})$. This heart has finite length and precisely $n$ simples, which have a canonical ordering, say $s_0,\ldots, s_{n-1}$ in which $\Ext{\cat{C}}{i}{s_j}{s_k}=0$ unless $j-k=i \geq 0$. Proposition 3.5 of \cite{MR2142382} states that, for $1\leq i \leq n-1$,
$$
L_{s_i}\cat{A}(e_0, \ldots, e_{n-1}) = \cat{A}\left(\sigma_i(e_0, \ldots, e_{n-1})\right)
$$
where $\sigma_i$ is the $i$th standard generator of the braid group. In particular the tilted heart $L_{s_i}\cat{A}(e_0, \ldots, e_{n-1})$ also arises from an excellent collection and so we can repeat the process.  Note, that this theory tells us nothing about the tilt at the first simple $s_0$, so that it only guarantees a partial version of Assumption~\ref{assumption1} (and does not help in verifying Assumption~\ref{assumption2} at all). A more refined situation is discussed in \cite[\S4]{MR2142382}, in which one has control over tilts at all simples so that Assumption~\ref{assumption1} holds, but this does not apply to either of the examples in the next section.

\section{Examples}
\label{examples}

Our main example, in \S\ref{constructible}, is the constructible derived category of $\P^1$ stratified by a point and its complement. The perverse sheaves are the heart of a $t$-structure on this category which satisfies Assumption~\ref{assumption2}, although this is only apparent after some computation. We therefore obtain a combinatorial description of a component of the space of stability conditions as a locally-finite tiling by subsets $U(\cat{A})\cong \U^2$. This description enables us to show that this component is isomorphic to $\C^2$ as a complex manifold, and to compute the subgroup of automorphisms of the category which fix this component.

As a counterpoint we consider stability conditions on the coherent derived category of $\P^1$ in \S\ref{coherent}. The space of these was computed in \cite{MR2219846} and is also isomorphic to $\C^2$. The Kronecker heart satisifes Assumption~\ref{assumption1} but not the stronger Assumption~\ref{assumption2}. The simple tilting process starting from this heart leads to a tiling of a dense subset of the space of stability conditions, but not of an entire component. The tiling is very different from that for the constructible sheaves, for example it is not locally-finite. This is a reflection of the fact that the Kronecker heart is of tame type (\ie has one-dimensional families of indecomposables), rather than of finite type as are the hearts arising in the constructible derived category. 

\subsection{Constructible sheaves on $\P^1$}
\label{constructible}

Let $X$ be $\P^1$ stratified by a point $x$ and its complement $U$ and let $\constr{X}$ be the constructible (with respect to this stratification) derived category of 
sheaves of complex vector-spaces on
 $X$. The first task is to construct some $t$-structures on $\constr{X}$. 
Both strata are simply-connected so the only $t$-structures on $\constr{x}$ and $\constr{U}$ are the usual ones (with heart the local systems) and their shifts. Let $\imath:x \hookrightarrow X$ and $\jmath: U \hookrightarrow X$ be the inclusions. Then there are functors
$$
\constr{x} \stackrel{\imath_*}{\longrightarrow} \constr{X} \stackrel{\jmath^*}{\longrightarrow} \constr{U} 
$$
with respective left and right adjoints $\imath^*$ and $\imath^!$ and $\jmath_!$ and $\jmath_*$ obeying the usual identities, see \eg \cite[\S1.4]{bbd}. Thus we have `gluing data' and given a \defn{perversity}, \ie a pair $(m,n) \in \Z^2$ indexing $t$-structures on $\constr{U}$ and $\constr{x}$, we can construct a $t$-structure on $\constr{X}$ with
$$
\constr{X}^{\leq 0} = \{ c \ | \ H^k(\jmath^*c)=0 \ \textrm{for} \ k>-m \ \textrm{and}\  H^k(\imath^*c)=0 \ \textrm{for} \ k>-n \}
$$
see \cite[\S1.4.9]{bbd}. Denote the heart of this $t$-structure by $\perv{n-m+1}{n}{}$. Verdier duality acts on these $t$-structures via $(m,n) \mapsto (2-m,-n)$ so that $D\perv{r}{s}{} = \perv{-r}{-s}{}$.  The category $\perv{r}{s}{}$ has finite length and two simple objects (one for each stratum). The simple objects in $\perv{r}{s}{}$ are $\C_x[s]$ and
\begin{eqnarray*}
\jmath_!\C_U[s-r+1] & \textrm{if} &  r>0\\
\C_X[s-r+1]& \textrm{if} & r=0 \\
\jmath_*\C_U[s-r+1]& \textrm{if} & r<0.
\end{eqnarray*}
For $-1 \leq r \leq 1$ this follows from \cite[Chapter 7, Proposition 1.10.1]{gema}. For $r\geq 2$ if $a\in \perv{r}{s}{}$ then the cohomological vanishing conditions imply $\Ext{}{1}{\imath_*\imath^*a}{\jmath_!\jmath^!a}=0$ so that the triangle
$$
\jmath_!\jmath^!a \to a \to \imath_*\imath^*a\to \jmath_!\jmath^!a[1]
$$
splits to give $a \cong \jmath_!\jmath^!a \oplus \imath_*\imath^*a$. It is then clear that the simple objects are as described. A Verdier dual argument applies in the case $r\leq -2$.

The category $\perv{r}{s}{}$ is semi-simple if $|r|>1$. When $|r|=1$ the heart is (up to a shift) either the category of constructible sheaves on $X$ (when $r=1$) or its Verdier dual (when $r=-1$). In either case the heart is equivalent to the representation category of the quiver $\cdot \to \cdot$ with two vertices and one arrow. The interesting case (the only one in which the $t$-structure is faithful) is when $r=0$ in which the heart is, up to a shift, the category $\cat{P}^0$ of perverse sheaves on $X$. This is equivalent to the category of representations of the quiver with relations $Q$: 
\begin{equation}
\label{quiver}
\xymatrix{\cdot \ar@/{}_{.5pc}/[r]_c& \cdot \ar@/{}_{.5pc}/[l]_v} \qquad vc=0,
\end{equation}
see, for example, \cite[Chapter 7, \S2.6]{gema}. The Auslander--Reiten quiver of the representation category is shown below. Its vertices are the (isomorphism classes of) the indecomposable representations and its arrows correspond to irreducible maps between them, \ie to maps which cannot be written as composites except where one factor is an isomorphism. The dotted lines indicate the almost-split short exact sequences. See \cite[Chapter VII]{MR1314422} for more details on Auslander--Reiten quivers. 
$$
\xymatrix{
&& p \ar[dr] &&\\
& \jmath_! \C_U[1] \ar[dr] \ar[ur] \ar@{--}[rr] && \jmath_*\C_U[1] \ar[dr] &\\
\C_x \ar[ur] \ar@{--}[rr]&& \C_X[1] \ar[ur]  \ar@{--}[rr]&& \C_x.
}
$$
The simple object $\C_X[1]$ is spherical of dimension $2$, and furthermore $$\ext{i}{\C_X[1]}{\C_x} = 0 = \ext{i}{\C_x}{\C_X[1]}$$ unless $i=1$. Hence, by the remarks on page \pageref{spherical remarks}, the twist $\Delta=\Phi_{\C_X[1]}$ satisfies
$$
\Delta \cat{P}^0 = L_{\C_X[1]} \cat{P}^0.
$$
\begin{remark}
We write $\Delta$ for this twist because geometrically it arises from a Dehn twist. Nadler and Zaslow \cite{Nadler:2006kx} construct an equivalence
$$
\constran{M} \simeq \derfuk{T^*M}
$$ 
between the real-analytically constructible derived category of a manifold $M$ and the derived asymptotic Fukaya category of its cotangent bundle. (The objects of the asymptotic Fukaya category are Lagrangians which are Legendrian at $\infty$, \ie asymptotic to a union of cotangent fibres, hence the name.) The equivalence takes a constructible sheaf to a Lagrangian smoothing of its characteristic cycle. Taking $M=X$ we can restrict to obtain an equivalence between $\constr{X}$ and the full subcategory of $\derfuk{T^*\P^1}$ generated by the cotangent fibre $T_x^*\P^1$ and zero section, which are the characteristic cycles of $\C_x$ and $\C_X$ respectively. Dehn twisting about the zero section gives an automorphism of $\derfuk{T^*\P^1}$. Algebraically this automorphism is given by twisting about the corresponding object $\C_X[1]$ in $\constr{X}$. 
\end{remark}

By direct computation we can check that
\begin{enumerate}
\item $\Delta$ has infinite order;
\item $D\Delta D \cong \Delta^{-1}$;
\item $\Delta\perv{-r}{-r}{} = \cat{P}^{r}$ when $r>0$.
\end{enumerate}
Using these properties we can show that the set of hearts $ \Delta^n\perv{r}{s}{}$ for $n,r,s, \in \Z$ is closed under simple tilts. All of these hearts are distinct, except for those identified by the third property.  Each of these hearts is either semi-simple, isomorphic to the constructible sheaves or to the perverse sheaves. Thus Assumption~\ref{assumption2} is satisfied. By Theorem~\ref{component thm} this is a list of the hearts of stability conditions in a component of the space of stability conditions.

\label{diagram explanation}
We describe the closure $\overline{U(\cat{P}^0)}$ of the set of stability conditions with heart the perverse sheaves (other cases are similar). The strata of codimension $\leq 1$ are described by Proposition~\ref{tiling};  we focus on the strata of codimension $2$. By Proposition~\ref{constraints} points of these correspond to degenerations $\mathcal{Z}$ of the central charge for which the charges of both simples become real, and the charge of no semi-stable vanishes. These limiting central charges correspond to points in the real plane --- coordinates $\mathcal{Z}(\C_x)$ and $\mathcal{Z}(\C_X[1])$ --- with certain points deleted, see Figure~\ref{bdy picture}. The deleted points form half-lines; we label these with the object which is semi-stable for stability conditions with nearby central charges and whose charge vanishes on that half-line. 

Each of the codimension $2$ strata in $\overline{U(\cat{P}^0)}$ is the (unique) codimension $2$ stratum in $U(\cat{B})$ for some right tilt $\cat{B}=R_\cat{T}\cat{P}^0$, and each such right tilt occurs. In Figure~\ref{bdy picture} the half-lines divide the plane into regions corresponding to codimension $2$ strata and each is labelled by the corresponding right tilt. Hearts labelling adjacent regions are related by simple tilts --- the transition in the direction of the arrows is given by right tilting at the object labelling the corresponding half-line. In particular the two simples in the heart corresponding to a region are the two objects labelling the bounding half-lines, but with objects on inward pointing arrows shifted by $[1]$. Explicit stability conditions with central charges and hearts as indicated can be constructed, thus verifying the results of \S\ref{main results} in this case. 

\begin{figure}[htbp]
\begin{center}
\includegraphics[width=.9\linewidth]{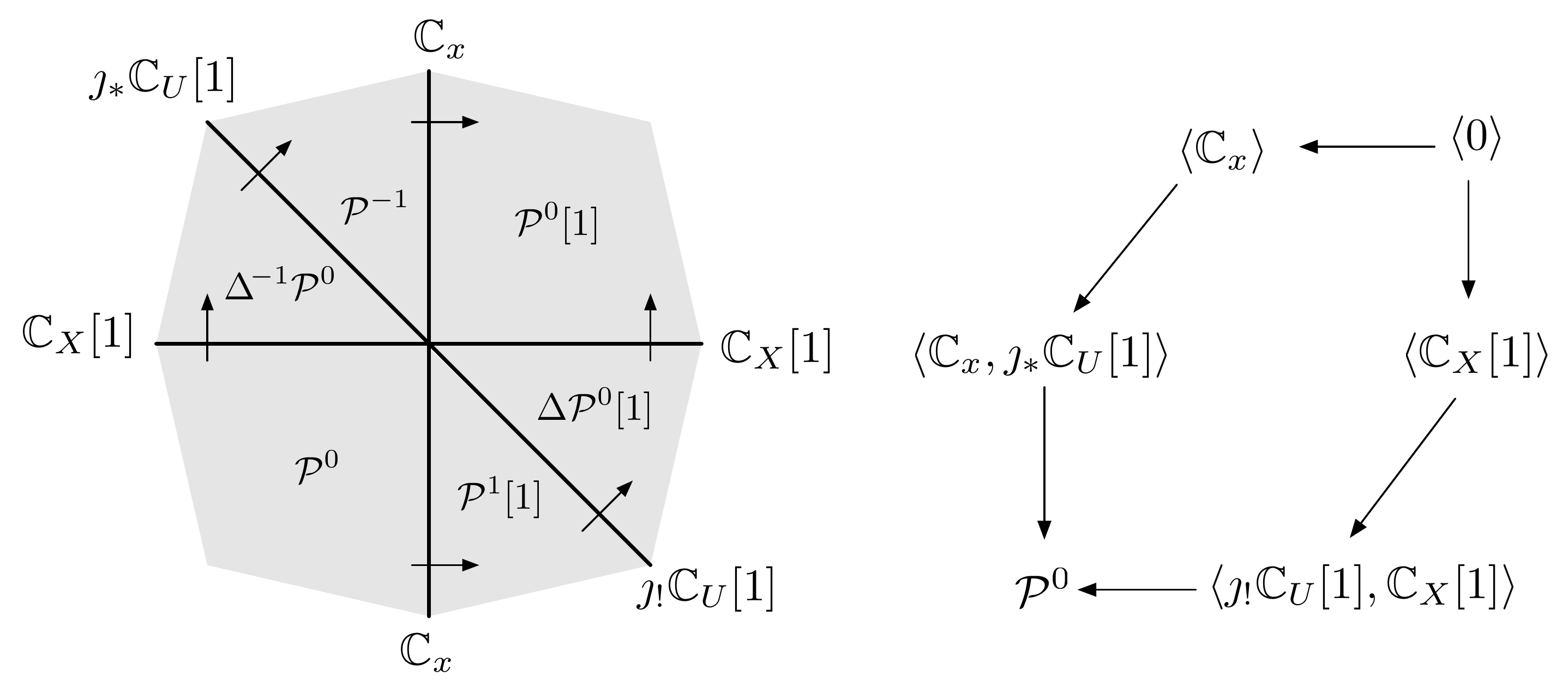}
\caption{The left-hand diagram shows the codimension $2$ strata in $\overline{U(\cat{P}^0)}$. The right-hand diagram shows the torsion theories in $\cat{P}^0$.  }
\label{bdy picture}
\end{center}
\end{figure}

A combinatorial description of the strata in the component $\stabo{\constr{X}}$ containing $U(\cat{P}^0)$ of the space of stability conditions on $\constr{X}$ is shown in Figure~\ref{graph}. 
\begin{theorem}
\label{c2thm}
There is a free action of $\C\times \langle \Delta\rangle \cong \C\times \Z$ on the component $\stabo{\constr{X}}$ of the space of stability conditions, where $\C$ and $\Delta$ act as described on p\pageref{c action}. The quotient is isomorphic to $\C^*$ and $\stabo{\constr{X}} \cong \C^2$.
\end{theorem}
\begin{proof}
The method is similar to that in \cite[\S 4]{MR2219846}. Since  $\Delta$ has infinite order and does not preserve any of the hearts it acts freely. The action of $\C$ is free and preserves semi-stables, whereas $\Delta$ changes them, so the action of $\C\times \langle\Delta\rangle$ is also free. Each orbit  contains a unique stability condition in the subset satisfying
\begin{enumerate}
\item $\C_x$ and $\jmath_!\C_U[1]$ are semi-stable;
\item $\mathcal{Z}(\C_x) = -1$ and $\varphi(\C_x) = 1$;
\item $\varphi(\jmath_!\C_U[1]) > 1$ or $\varphi(\jmath_!\C_U[1])=1$ and the mass $m(\jmath_!\C_U[1]) \in (0,1)$. 
\end{enumerate}
To see this note that any orbit  intersects $\bigcup_{i\in \N} U(\cat{P}^i)$.
For $i\geq 1$ both $\C_x$ and a shift of $\jmath_!\C_U[1]$ are simple in $\cat{P}^i$ and hence semi-stable. For the same reason $\C_x$ is semi-stable in $U(\cat{P}^0)$ but in this subset $\jmath_!\C_U[1]$ is semi-stable if and only if $$\varphi(\jmath_!\C_U[1]) \geq \varphi(\C_x).$$ However, if $\varphi(\jmath_!\C_U[1]) < \varphi(\C_x)$ then applying $\Delta$ we get a stability condition in $U(\Delta\cat{P}^0)$ for which $\varphi(\jmath_!\C_U[1]) > \varphi(\C_x)$ and it follows that $\C_x$ and $\jmath_!\C_U[1]$ are semi-stable. Thus each orbit contains a stability condition for which $\C_x$ and $\jmath_!\C_U[1]$ are semi-stable. Using the action of $\C$ we can rotate and rescale so that the second condition holds. 

Conversely, when the first condition is satisifed the heart of the stability condition must be one of the $\cat{P}^i$ for $i\geq 0$ or $\Delta \cat{P}^0$. Furthermore if it is $\cat{P}^0$ or $\Delta \cat{P}^0$ then 
$$
\varphi(\jmath_!\C_U[1]) \geq \varphi(\C_x)
$$
(otherwise $\jmath_!\C_U[1]$ is not semi-stable). It follows that $\varphi(\jmath_!\C_U[1])\geq 1$ when the first two conditions are satisifed. For phases $>1$ there are no restrictions. If the phase is $1$ then the heart is $\Delta \cat{P}^0$ when $m(\jmath_!\C_U[1])\in (0,1)$ and $\cat{P}^0$ when $m(\jmath_!\C_U[1])\in (1,\infty)$. The case $m(\jmath_!\C_U[1]) = 1$ does not occur for this implies $\mathcal{Z}(\C_X)=0$ which is forbidden since $\C_X$ is semi-stable. Furthermore, if $m(\jmath_!\C_U[1])=m\in (1,\infty)$ then acting by $(-\frac{1}{\pi i} \log(m),\Delta)$ we obtain a new stability condition, again satisfying the first two conditions and with $\varphi(\jmath_!\C_U[1])= 1$, but with $m(\jmath_!\C_U[1])=1/m \in (0,1)$. Hence, as claimed, each orbit contains a stability condition satisfying all three conditions. It is easy to check that no orbit meets the subset satisfying the three conditions in more than one point.

We identify $\stabo{\constr{X}} / \C \times \langle\Delta\rangle$ with $\C^*$ by first mapping a stability condition satisfying the three conditions to 
$$
\log m( \jmath_!\C_U[1] ) + i \pi \varphi(\jmath_!\C_U[1] ) 
\in \{ x+iy \in \C \ | \ y>  \pi \ \textrm{or}\  y=\pi, x<0 \}
$$
and then applying $z \mapsto (z-i\pi)^2$. This is holomorphic because the complex structure on  a space of stability conditions comes from that on the space of central charges. Since $\jmath_!\C_U[1-i]$ is simple in $\cat{P}^i$ when $i>0$ we can choose $\sigma\in U(\cat{P}^i)$ satisfying the three conditions for which the mass and phase of $\jmath_!\C_U[1]$ take any given values in $(0,\infty)$ and $(i,i+1]$ respectively. We can also find $\sigma \in U(\Delta\cat{P}^0)$ satisfying the three conditions such that $\jmath_!\C_U[1]$ has phase $1$ and any given mass in $(0,1)$. It follows that the above map is surjective. Furthermore, the induced gluing of the boundary is exactly that arising from the group action. 

Finally, since $\stabo{\constr{X}}/\C$ is connected it must be the universal cover of $\C^*$, with $\Delta$ acting as deck transformations, and is therefore isomorphic to $\C$. The result follows. 
\end{proof}

\begin{figure}[htbp]
\begin{center}
\includegraphics[width=2in, angle=-90]{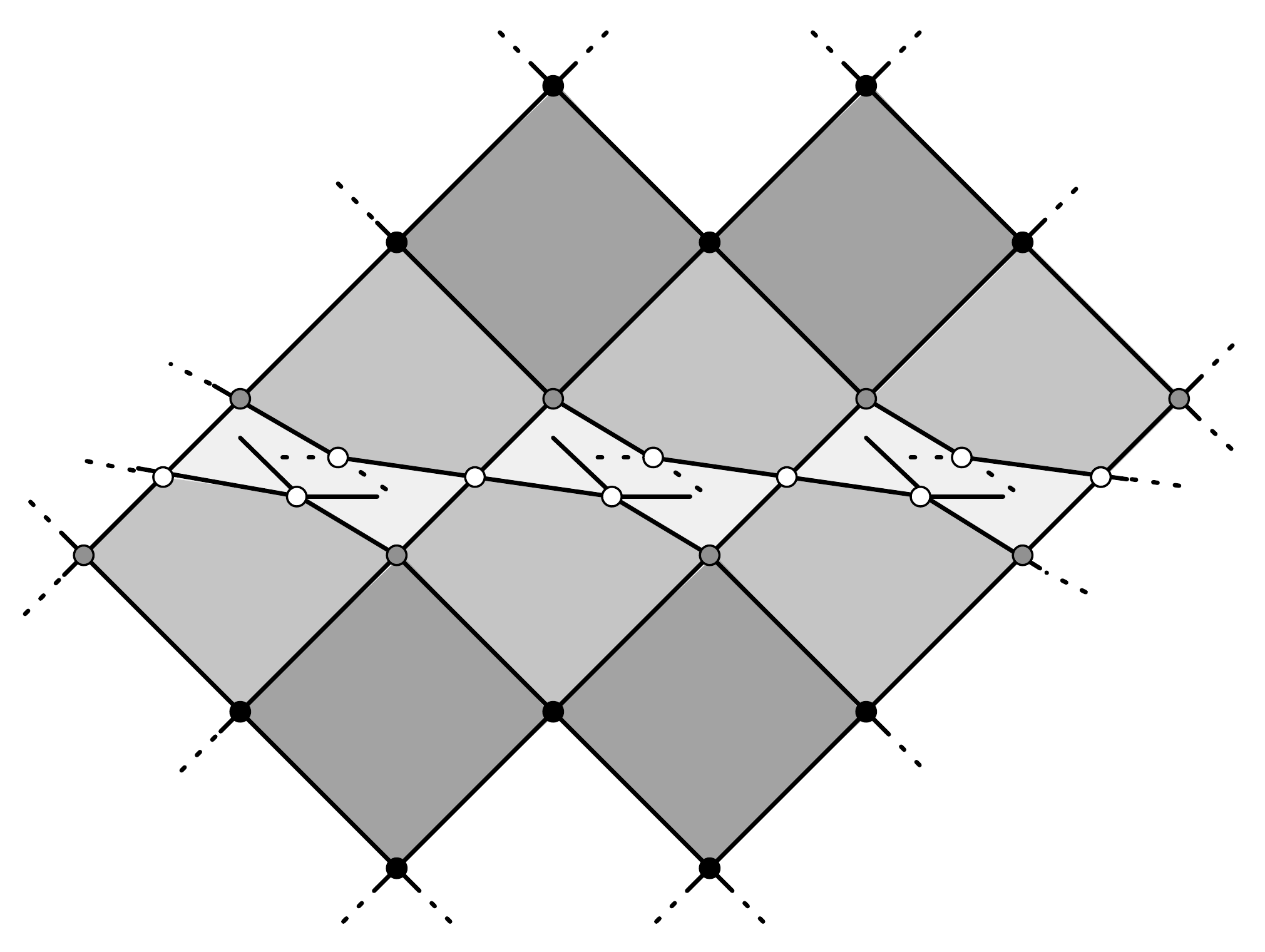}
\caption{The Poincar\'e dual to the stratification of $\stabo{\constr{X}}$. Vertices correspond to open strata, equivalently to the hearts of $t$-structures. Vertices labelled by white, grey and black dots correspond to hearts equivalent to the perverse sheaves, to the constructible sheaves and to a semi-simple category respectively. Adjacent vertices are related by simple tilts, equivalently if the corresponding subsets of stability conditions share a codimension $1$ boundary stratum. Moving right corresponds to tilting right, moving left to tilting left. The $2$-cells correspond to codimension $2$ strata in the closure of the strata corresponding to their vertices and edges. The diagram shows part of one `sheet' which extends to infinity as indicated. There are countably many such sheets, joined along the central `spine', attached so that one can pass from the lower part shown to the sheet above the upper part, and so on. }
\label{graph}
\end{center}
\vspace{.25in}
\begin{center}
\includegraphics[width=2.75in]{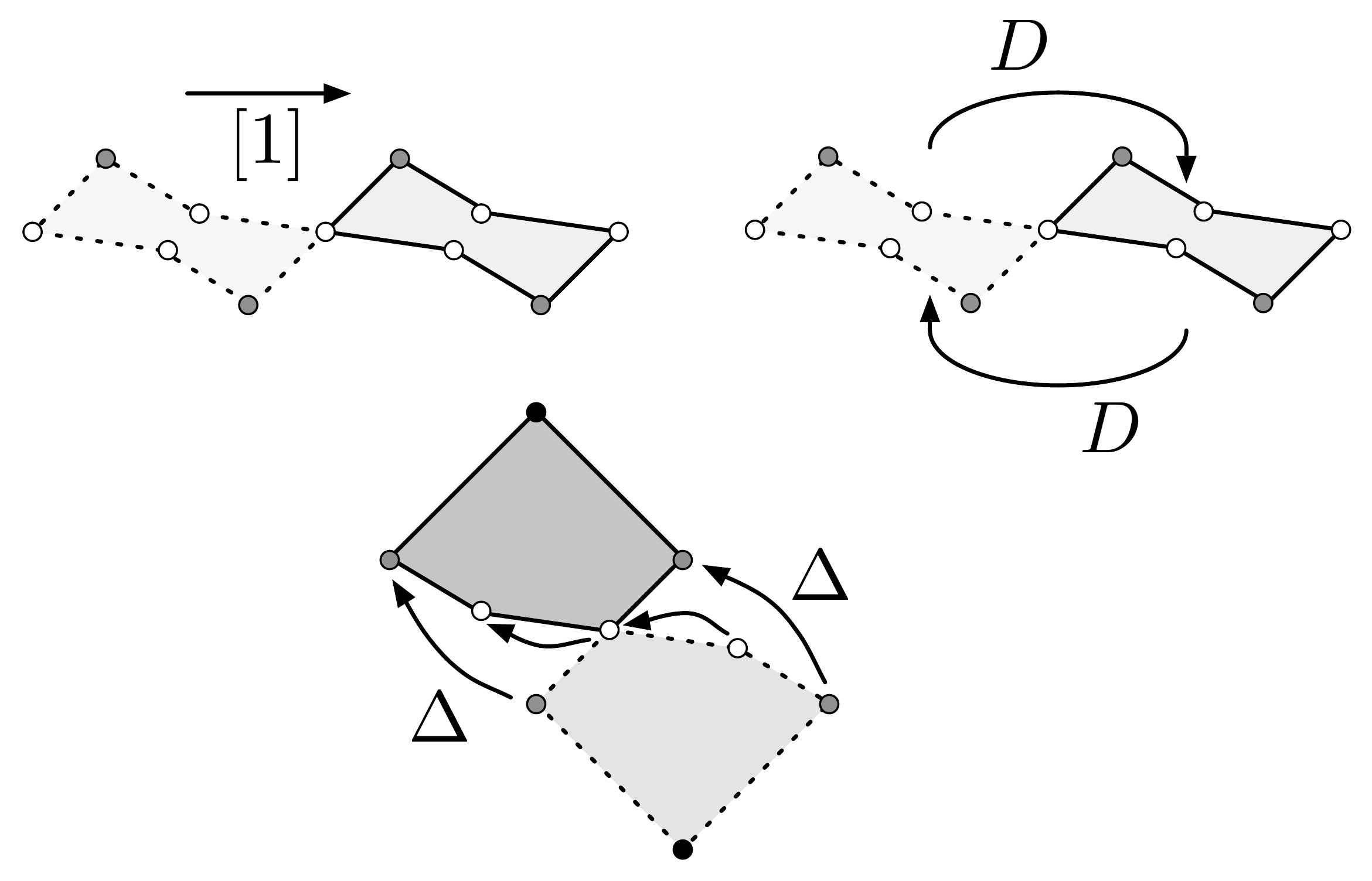}
\caption{The actions of shift, $\Delta$ and Verdier duality on Figure~\ref{graph}. Shift moves the diagram one place to the right. Duality rotates it by $\pi$ about the perverse sheaves, swapping the upper and lower layers, and $\Delta$ rotates by $\pi$ about the central spine and shifts by half a place to the left. Shift and $\Delta$ preserve tilting relationships whereas duality reverses them. }
\label{actions}
\end{center}
\end{figure}

As an application we determine the subgroup of the automorphism group of the category which preserves the component $\stabo{\constr{X}}$. (It is conjectured that spaces of stability conditions are connected; if this were known we could determine the entire automorphism group.) It is convenient to determine the subgroup of $\delta$-automorphisms first, \ie to allow functors such as $[1]$ as well as triangulated automorphisms. Suppose $\alpha$ is a $\delta$-automorphism. Since $\alpha$ must map $U(\cat{A})$ to $U(\alpha\cat{A})$ we can check from the combinatorial description that $\alpha\Delta^m[n]$ preserves the perverse sheaves for some $m$ and $n$. An automorphism which preserves the perverse sheaves induces an action on the Auslander--Reiten quiver. In this case it must act as the identity on the vertices, as there are no other symmetries; up to a natural isomorphism we may assume it fixes the indecomposable perverse sheaves. Since there is at most one irreducible map between each pair of indecomposables, the automorphism can only act by rescaling each of these maps by an element of $\C^*$. By a further natural isomorphism we can assume these rescalings are all trivial except for one, for concreteness say that of the map $\jmath_*\C_U[1] \to \C_x$. (We use the fact that the sum of the two composites from $\jmath_!\C_U[1]$ to $\jmath_*\C_U[1]$ is zero here.) Therefore the subgroup of automorphisms (up to natural isomorphism) which preserve the perverse sheaves is isomorphic to $\C^*$. Hence 
$$
\textrm{$\delta$-Aut}\, \constr{X} \cong \Z^2 \times \C^*
$$
with $\C^*$ acting trivially on the space of stability conditions and $\Z^2$ generated by $\Delta$ and $[1]$. The {\it bona fide} automorphisms are the subgroup $\Z \times 2\Z \times \C^*$. This description makes it easy to identify the Serre functor; it is $S= \Delta^{-2}$.

\subsection{Coherent sheaves on $\P^1$}
\label{coherent}

Let $\cat{D}(\P^1)$ be the coherent derived category of $\P^1$. The space of stability conditions was computed in  \cite{MR2219846} and is isomorphic to $\C^2$. We use the notion of excellent collection --- see the remarks on page \pageref{excellent remarks} and  \cite[\S3]{MR2142382} --- in this case $\mathcal{O}, \mathcal{O}(1)$. The corresponding heart  $\cat{A}(\mathcal{O}, \mathcal{O}(1))$ is equivalent to the representation category of the Kronecker quiver
\begin{equation}
\label{kronecker}
K= \xymatrix{\cdot \ar@/{}_{.5pc}/[r] \ar@/{}^{.5pc}/[r] & \cdot }
\end{equation}
and we refer to it as the Kronecker heart. It has two simple objects $\mathcal{O}$ and $\mathcal{O}(-1)[1]$ (in that order in the canonical ordering). Right tilting at the first leads to the heart corresponding to the mutated collection:
$$
\cat{A}(\mathcal{O}(1),\mathcal{O}(2)) = R_{\mathcal{O}} \cat{A}(\mathcal{O}, \mathcal{O}(1)).
$$
Write $\cat{A}_d$ for $\cat{A}(\mathcal{O}(d),\mathcal{O}(d+1))$. Repeating we see that
$\cat{A}_{d+1} = R_{\langle \mathcal{O}, \ldots,\mathcal{O}(d)\rangle} \cat{A}_0$. Hence by Corollary~\ref{intersection corollary}
$$
U(\cat{A}_d ) \cap \overline{U(\cat{A}_0)} \neq \emptyset
$$
and in fact is of (real) dimension two. This yields the picture of the (infinitely many) codimension two strata in $\overline{U(\cat{A}_0)}$ shown in Figure~\ref{kronecker bdy}. 
 
\begin{figure}[htbp]
\begin{center}
\includegraphics[width=\linewidth]{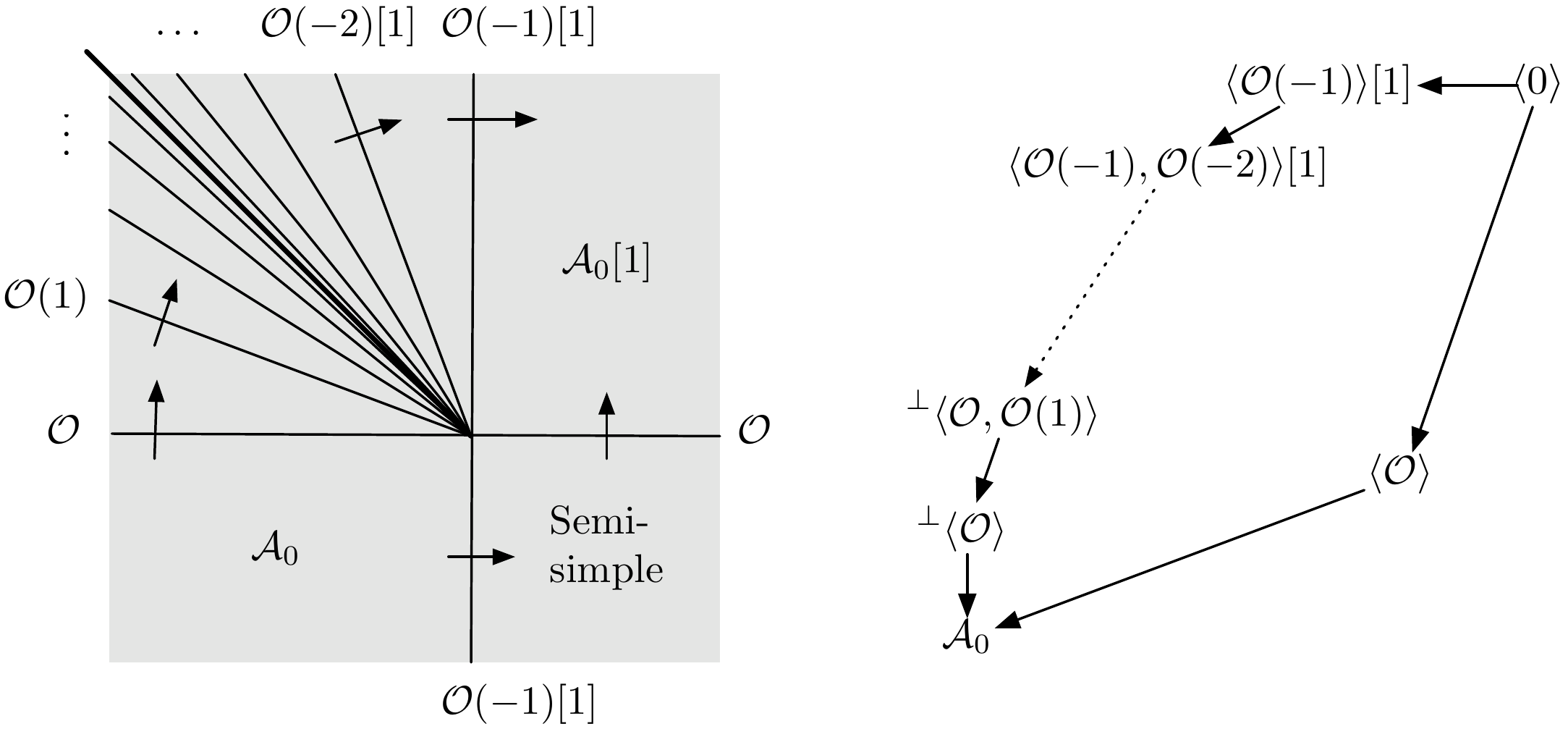}
\caption{The codimension $2$ strata in $\overline{U(\cat{A}_0 )}$ where $\cat{A}_0$ is the Kronecker heart, and the related torsion theories in $\cat{A}_0$. The interpretation is analogous to that of Figure~\ref{bdy picture}.}  
\label{kronecker bdy}
\end{center}
\end{figure} 

Assumption~\ref{assumption2} is false for the Kronecker heart, and we do not obtain an entire component of the space of stability conditions by tilting from it. There is a stability condition with heart  $\mathrm{Coh}(\P^1)$ and central charge $\mathcal{Z}(\mathcal{E}) = -\deg(\mathcal{E}) + i\, \mathrm{rank}(\mathcal{E})$.  The semi-stables of phase $\varphi$ are the semi-stable (in the usual sense) coherent sheaves of slope $-1/\tan(\pi\varphi)$.  Rotating this central charge we obtain stability conditions in $U(\cat{A}_0)$ \cite[Proposition 2.4]{MR2219846}. It follows that these stability conditions are in the same component. However, the coherent sheaves are not obtained from the Kronecker heart by any \emph{finite} sequence of simple tilts, although they are the left tilt of it at the torsion theory $\langle \mathcal{O}(-d)[1] \ | \ d\geq 1 \rangle$. This torsion theory contains infinitely many indecomposables, but it is still true that 
$$
U(\cat{A}_0)\cap \overline{U(\mathrm{Coh}(\P^1))} \neq \emptyset
$$
and in fact has the expected codimension, two. To see this we note that the standard stability condition on the coherent sheaves degenerates to one in which  the phase of $\mathcal{O}(d)$ is $0$ for $d<0$ and $1$ for $d\geq 0$ (and all $\mathcal{O}(d)$ remain semi-stable). The heart is therefore the Kronecker heart. We can freely choose the masses of $\mathcal{O}$ and $\mathcal{O}(-1)[1]$ in this degeneration and so obtain a (real) dimension two intersection as claimed. 

It is not however true that Corollary~\ref{intersection corollary converse} holds without any assumptions on the torsion theory. From Figure~\ref{kronecker bdy} we can see that
$$
U(\mathrm{Coh}(\P^1)[1])\cap \overline{U(\cat{A}_0)} = \emptyset
$$
even though $\cat{A}_0 = L_\cat{T} \mathrm{Coh}(\P^1)[1]$ for $\cat{T}={}^\perp\langle \mathcal{O}(-d)[1]\ |\ d\geq 1\rangle$. The degenerations of central charges which `should' give stability conditions in the intersection are forbidden because the charges of semi-stable (shifted) torsion sheaves vanish.

If we allow hearts obtained from a given one by any finite sequence of tilts (rather than just simple tilts) then we obtain an entire component of the space of stability conditions in this example. However, some tilts lead to hearts which are not the heart of any stability condition, see \cite[Remark 3.5]{MR2219846}. Proposition~\ref{constraints} still holds --- degenerations of the central charge for which the charge of no semi-stable vanishes lift to degenerations of stability conditions --- even though the given proof is invalid.

\end{document}